\RequirePackage[leqno]{amsmath}
\documentclass[leqno,11pt]{amsart}
\usepackage[utf8]{inputenc}
\usepackage[T1]{fontenc}
\usepackage[dvipsnames]{xcolor}
\usepackage{amsmath,amsthm,amssymb}
\usepackage{newunicodechar}
\usepackage{amsmath}
\usepackage{cases}
\usepackage{multicol}
\usepackage{enumerate}
\usepackage{MnSymbol,wasysym}
\usepackage{graphicx, multicol, latexsym, amsmath, amssymb}
\usepackage{amsfonts}
 \usepackage{hyperref}
 \hypersetup{%
  colorlinks=true,
  linkcolor=black,
  citecolor=black,
  urlcolor=blue,
  pdftitle={A Keller-Segel type taxis model with ecological interpretation and boundedness due to gradient nonlinearities},
  pdfauthor={Ishida, Lankeit, Viglialoro},
  pdfkeywords={},
  bookmarksopen=false,
}
\usepackage{empheq}
\usepackage[T1]{fontenc}
\usepackage{tabularx}
\usepackage[leftcaption]{sidecap}
\sidecaptionvpos{figure}{m}
\usepackage{tikz}
\usepackage{subfigure}
\usepackage{geometry}
\usetikzlibrary{backgrounds,automata}
\usepackage[inline]{enumitem}

\makeatletter
\@namedef{subjclassname@2020}{\textup{2020} Mathematics Subject Classification}
\makeatother

\def\R{\mathbb R}

\def\R{\mathbb R}  
\def\TM{T_{max}} 
\def
\@cite
#1#2{[{{\bfseries #1}\if@tempswa , #2\fi}]}

\newunicodechar{ℕ}{\mathbb{N}}
\newunicodechar{ℝ}{\mathbb{R}}
\newunicodechar{Δ}{\Delta}
\newunicodechar{∇}{\nabla}
\newunicodechar{∞}{\infty}
\newunicodechar{→}{\to}
\newunicodechar{∂}{\partial}
\newunicodechar{ν}{\nu}
\newunicodechar{Φ}{\Phi}
\newunicodechar{δ}{\delta}
\newunicodechar{τ}{\tau}
\newunicodechar{α}{\alpha}
\newunicodechar{β}{\beta}
\newunicodechar{γ}{\gamma}
\newunicodechar{χ}{\chi}
\newunicodechar{μ}{\mu}

\newcommand{\Om}{\Omega} 
\newcommand{\norm}[2][]{\|#2\|_{#1}}
\newcommand{\normm}[2]{\|#2\|_{#1}}
\newcommand{\set}[1]{\{#1\}}
\newcommand{\kl}[1]{\left(#1\right)}
\newcommand{\f}[2]{\frac{#1}{#2}}
\newcommand{\upto}{\nearrow}
\newcommand{\Tmax}{T_{max}}

\newcommand{\ubar}{\overline{u}}
\newcommand{\Ombar}{\overline{\Om}}
\newcommand{\Lom}[1]{L^{#1}(\Omega)}
\newcommand{\nn}{\nonumber}
\newcommand{\bdry}{\vert_{∂\Om}}
\newcommand{\io}{\int_\Omega}
\newcounter{Ccnt}
\makeatletter
\newcommand\C[1]{%
	\@ifundefined{C-#1}%
	{\stepcounter{Ccnt}\expandafter\xdef\csname C-#1\endcsname{\arabic{Ccnt}}}%
	{}%
	c_{\csname C-#1\endcsname}}
\makeatother
\newtheorem{theorem}{Theorem}[section]
\newtheorem{corollary}[theorem]{Corollary}
\newtheorem{lemma}[theorem]{Lemma}

\newtheorem{remark}{Remark}

\newcommand\scalemath[2][1.0]{\scalebox{#1}{\mbox{\ensuremath{\displaystyle #2}}}}
\title[B\MakeLowercase{oundedness in a chemotaxis model with gradient terms}] 
      {
\huge{A K\MakeLowercase{eller}--S\MakeLowercase{egel type taxis model with ecological interpretation and boundedness due to gradient nonlinearities}}}
\author[S\MakeLowercase{achiko} I\MakeLowercase{shida}, J\MakeLowercase{ohannes}  L\MakeLowercase{ankeit and}   G\MakeLowercase{iuseppe} V\MakeLowercase{iglialoro}]{}
\keywords{Chemotaxis, Ecology, Global existence, Boundedness, Gradient nonlinearities. \\
\textit{$^\sharp$Corresponding author}: lankeit@ifam.uni-hannover.de}
\begin{document}
\maketitle
\centerline{\LARGE S\MakeLowercase{achiko} I\MakeLowercase{shida}$^a$, J\MakeLowercase{ohannes} L\MakeLowercase{ankeit}$^{b,\sharp}$ \and G\MakeLowercase{iuseppe} V\MakeLowercase{iglialoro}$^{c}$}
\medskip
{   \small
 \medskip
\centerline{$^a$Department of Mathematics and Informatics, 
 Graduate School of Science, Chiba University}
\centerline{1-33, Yayoicho, Inage-ku, Chiba-shi, Chiba 263-8522 (Japan)}
 \medskip
 \centerline{$^b$Leibniz Universität Hannover,
 Institut für Angewandte Mathematik}
 \centerline{Welfengarten 1, 
30167 Hannover (Germany)}
\centerline{ORCID: 0000-0002-2563-7759}
 \medskip
 \centerline{$^c$Dipartimento di Matematica e Informatica,
Universit\`{a} di Cagliari}
 \centerline{Via Ospedale 72, 09124. Cagliari (Italy)}
 \centerline{ORCID: 0000-0002-2994-4123}
}
\bigskip
\begin{abstract}
We introduce a novel gradient-based damping term into a Keller--Segel type taxis model with motivation from  
ecology and consider the following system equipped with homogeneous Neumann-boundary conditions: 
\begin{equation}\label{problem_abstract}
\tag{$\Diamond$}
\begin{cases}
u_t= \Delta u - \chi \nabla \cdot (u \nabla v)+a u^\alpha-b u^\beta-c|\nabla u|^\gamma  & \text{ in } \Omega \times (0,T_{max}),\\
\tau v_t=\Delta v-v+u  & \text{ in } \Omega \times (0,T_{max}).\\
\end{cases}
\end{equation}
The problem is formulated in a bounded and smooth domain $\Omega$ of $\R^N$, with $N\geq 2$, for some  positive numbers $a,b,c,\chi>0$, $\tau \in \{0,1\}$, $\gamma\geq 1$, $\beta>\alpha\geq 1$ and with $\TM\in (0,\infty]$. As far as we know, Keller--Segel models with gradient-dependent sources are new in the literature and, accordingly, beyond giving a reasonable ecological interpretation the objective of the paper is twofold:  
\begin{enumerate}[label=\roman*)]
\item \label{Abs:Obj-1} to provide a rigorous analysis concerning the local existence and exensibility criterion for a class of models generalizing problem \eqref{problem_abstract}, obtained by replacing $a u^\alpha-b u^\beta-c|\nabla u|^\gamma$ with $f(u)-g(\nabla u)$; 
\item \label{Abs:Obj-2} to establish sufficient conditions on the data of problem \eqref{problem_abstract} itself, such that it admits a unique classical solution $(u,v)$, for $\TM=\infty$ and with both $u$ and $v$ bounded. 
\end{enumerate}
We handle \ref{Abs:Obj-1} whenever appropriately regular initial distributions $u(x,0)=u_0(x)\geq 0$, $\tau v(x,0)=\tau v_0(x)\geq 0$ are considered and $f$ and $g$ obey some regularity properties and, moreover, some growth restrictions. Further, as to \ref{Abs:Obj-2}, for the same initial data considered in the previous case, global  boundedness of solutions is proven for any $\tau\in \{0,1\}$, provided that  $\frac{2N}{N+1}<\gamma\leq 2$. 
\end{abstract}

\section{Introduction, main claims and organization of the paper}\label{Intro}
\subsection{Taxis terms in ecological models}
The original Keller--Segel model 
\begin{equation}\label{KS}
 u_t = Δu -∇\cdot(u∇v), \qquad v_t=Δv-v+u
\end{equation}
was introduced to describe the self-organized aggregation in a starving slime mould colony (see \cite{K-S-1970, BBTW}), where individual cells (density $u$) direct their otherwise random movement toward higher concentrations of a chemical signal substance they produce (concentration $v$). 
On much larger scales yet, this taxis system and its relatives find their use, for example in modelling the motion of predator populations in response to the concentration of their prey (or vice versa), see e.g. \cite{kareiva_odell}.

But also if the second species is not explicitly modelled (that is, also if the second equation remains of the form in \eqref{KS} and does not incorporate Lotka--Volterra type interactions), and even in the absence of a chemical signal, 
it is possible to understand the second component in \eqref{KS} as portrayal of the tendency of animals to have their motion influenced by their spatial memory \cite{fagan2013memory}. 

If attraction to (or, with a different sign, repulsion from) previously visited sites is incorporated in a heat equation model of random motion by means of a so-called weak memory kernel, it is precisely \eqref{KS} that results, \cite{shi_shi_wu_spatial_memory}.

\subsection{Interaction terms in equations of population dynamics}
Especially in a setting involved with population dynamics, that is, $u$ denoting the density of a biological species occupying a certain domain, the combination of \eqref{KS} with terms describing population growth or decay is very natural, for example logistic growth \cite{verhulst} (in the context of chemotaxis models, see \cite{AiTsEdYaMi,MIMURATsujikawa}). In this case, the temporal evolution of the distribution of the species is described as interplay between displacements (in the form of diffusion and directed motion), reproduction and death, with the death rate increasing for large populations, so that the zero-order terms in $u_t = Δu - ∇\cdot(u∇v) + f(u)$ may have the form $f(u)=au^{α}-bu^{β}$, where $α=1$ if the birth rate is constant ($a$) and where $β>1$. 

Nevertheless, not only the current size of a population has an influence on its growth; beyond such density-mediated effects, also its behaviour can have a significant impact. Since it is not unusual for predators to react to motion (e.g. \cite{bednarski2012optical} 
or 
\cite{maturana1960anatomy}), 
for the prey population, motility can increase the risk to be decimated by attacks; 
see \cite{geipel2020predation} 
for supporting evidence in the case of katydids, or 
\cite{visser2007motility,almeda2017behavior} for studies of zooplankton, whose mortality rate depends on the swimming velocity. 

In a macroscopic model, such effects appear as additional decay terms depending on the size of the gradient of the solution, i.e. of the gradient of the population density, so that we obtain $u_t = Δu - ∇\cdot(u∇v) + f(u) - g(∇u)$, where $g$ is an increasing function of $|∇u|$
. (Indeed, this coincides with the interpretation of gradient terms in scalar parabolic equations as 'accidental deaths' suggested by Souplet in \cite{SoupBio}.)
\subsection{The model}
In conclusion, we aim to study the signal-production chemotaxis model 
\begin{equation}\label{problem}
\begin{cases}
u_t= \Delta u - \chi \nabla \cdot (u \nabla v)+ a u^\alpha-bu^\beta-c |\nabla u|^\gamma & \text{ in } \Omega \times \left(0,T_{max}\right),\\
\tau v_t=\Delta v-v+ u  & \text{ in } \Omega \times \left(0,T_{max}\right),\\
∂_{ν} u=∂_{ν}v=0 & \text{ on } \partial \Omega \times \left(0,T_{max}\right),\\
u(x,0)=u_0(x), \; \tau v(x,0)=\tau v_0(x) & x \in \bar\Omega,
\end{cases}
\end{equation}
where $a,b,c,\chi>0$, $1\leq \alpha<\beta$, $\gamma \geq 1$, $\tau\in \{0,1\}$, and where the problem is posed in a bounded and smooth domain $\Omega$ of $\R^N$, with $N\geq 2$. Moreover, sufficiently  regular initial data  $u_0\geq 0$, $\tau v_0\geq 0$ are given, 
$\partial_{ν}$ indicates the outward normal derivative on $\partial \Omega$, whereas $T_{max}\in (0,\infty]$ is the maximal time up to which solutions to the system are defined. (For a local existence result and a characterization of the maximal existence time for a class of models generalizing \eqref{problem} see Theorem \ref{localSol} below.)

The parameter $\tau\in \{0,1\}$ distinguishes between the 'fully parabolic' chemotaxis system ($τ=1$) and its 'parabolic-elliptic' counterpart ($τ=0$). The latter is a frequent simplification of the model, which is helpful or necessary in certain proofs, and from a biological perspective can be justified whenever the spread of $v$ occurs on a much faster time scale than that of $u$.

\subsection{Analysis of the Keller--Segel system }
Mathematically, \eqref{KS} is known for its ability to admit global bounded solutions and unbounded solutions blowing up in finite time, depending on the total mass of initial data and the specific initial configurations. More precisely, for $N=1$ all solutions are uniformly bounded in time, whereas for $N\ge 3$ and any positive mass $m$, initial data with this mass $m=\io u_0$ that lead to blow-up within finite time have been detected. On the other hand, when $N=2$, the total mass $m$ separates the case where diffusion overcomes self-attraction (if $m<4\pi$, solutions are global in time) from the case where self-attraction dominates and initial data giving rise to finite-time blow-up exist. An overview can be found in the surveys \cite{lan_win_survey, BBTW}, and we mention \cite{HerreroVelazquez,Nagai,win_aggregation_vs} as some notable results in this context. 

\subsection{The Keller--Segel system with logistic source terms}
While mathematical intuition indicates that the presence of logistic terms (and thus of a superlinear damping effect) should benefit the boundedness of solutions (paralleling the boundedness of solutions to the ODE $u'=au-bu^{β}$ for $β>1$) and prevent the extreme structure formation in the sense of finite-time blow-up found in \eqref{KS}, this has only been shown for large values of $b$  (if $β=2$, see \cite{TelloWinkParEl}, \cite{W0}; and for weak solutions featuring eventual smoothness see \cite{lankeit_evsmooth}), whereas the possibility of blow-up was established for certain values of $\beta>1$, first for dimension $5$ or higher \cite{WinDespiteLogistic}, when the equation for $v$ reads $0=\Delta v-\frac{1}{|\Omega|}\int_\Omega u(x,t)dx+u$ (see also \cite{FuestCriticalNoDEA} for a recent improvement of \cite{WinDespiteLogistic}), but later also in 3-dimensional domains (with $0=\Delta v-v +u$ and $β<\f76$), \cite{Winkler_ZAMP-FiniteTimeLowDimension}, with extensions to more general systems with nonlinear diffusion and sensitivity functions, \cite{black_fuest_lankeit,tanaka}.

\subsection{Gradient nonlinearities}
The study of gradient terms in parabolic equations, and their impact on the possibility of blow-up is a classical topic
(see, e.g. \cite{chipot_weissler,kawohl_peletier,fila}, where the relations between exponents $γ$ and $α$ were investigated that can prevent or admit blow-up in the Dirichlet problem for a semilinear heat equation of the form $u_t=Δu - c|∇u|^{γ} + u^{α}$, or \cite[Sec. IV]{QS}). Additionally, a possible biological interpretation of these terms was given in the above-mentioned \cite{SoupBio}. Nevertheless, in chemotaxis models, gradient nonlinearities of this kind seem to have been treated only in \cite{ViglialoroDifferentialIntegralEquations}, where a lower bound of the maximal existence time of solutions was derived. 

\subsection{Presentation of the main results}
A general existence theory or any result on boundedness of solutions, however, seem to be lacking for Keller-Segel type taxis models with gradient terms. In particular motivated by the observation that some generalized logistic sources still permit blow-up (see above: \cite{WinDespiteLogistic,Winkler_ZAMP-FiniteTimeLowDimension,black_fuest_lankeit,tanaka,FuestCriticalNoDEA}) in such models, we aim at deriving sufficient conditions such that the combined actions of the reproductive growth ($+au^\alpha$) and the two degradation mechanisms ($-bu^\beta-c|\nabla u|^\gamma$) may guarantee their global existence. It is our goal to pose the condition on the strength of the `accidental death`, i.e. the gradient term, and not the (zero-order) decay by `natural death` that was part of earlier works. 

A natural question we want to deal with thus is:
\begin{enumerate}[label=$\mathcal{Q}_\arabic*$]
\item\label{Question}: For $\tau=c=0$, $\alpha=1$ and $\beta<2$ as in \cite{Winkler_ZAMP-FiniteTimeLowDimension}, let $u_0$ be such that the related solution to \eqref{problem} blows at finite time. Do some conditions on $\gamma\geq 1$ exist so to ensure that, for $c>0$ only, such blow-up scenario is prevented? 
\end{enumerate}
In order to engage with this question, we first require a rigorous analysis concerning the local existence of model \eqref{problem}, which was not available in the literature.
As  consequence, we firstly address questions connected to existence of (local) classical solutions and related extensibility criteria applicable to a wider context than that specifically formulated in system \eqref{problem}. In particular, once $\Omega$, $\chi>0 $, $\tau\in\set{0,1}$ are given as indicated above, we will deal with this problem:  
\begin{equation}\label{system}
 \begin{cases}
  u_t = Δu - \chi ∇\cdot(u∇v) + f(u) - g(∇u)& \text{ in } \Omega \times \left(0,T_{max}\right),\\
  τv_t = Δv - v + u & \text{ in } \Omega \times \left(0,T_{max}\right),\\
∂_{ν}u=∂_{ν}v=0 & \text{ on } \partial \Omega \times \left(0,T_{max}\right),\\
u(x,0)=u_0(x), \; \tau v(x,0)=\tau v_0(x) & x \in \bar\Omega.
 \end{cases}
\end{equation}
Here $f$ and $g$ will suitably generalize the terms of \eqref{problem} and, more precisely 
may have to fulfill these assumptions:
\begin{subnumcases}{\label{cond:fandg}}
f\colon ℝ→ℝ \textrm{ is locally Lipschitz continuous, with } f(0)\ge 0; \label{fAssTheo} \\
\textrm{ there is } C_f>0  \textrm{ such that } f(s)\leq C_f \quad \textrm{for all } \; s\geq 0;\label{ConstantForfThe}  \\ 
g\colon ℝ^N\to [0,∞) \textrm{ is locally Lipschitz continuous, with } g(0)= 0; \label{GlocallyLips} \\
\textrm{ there is } C_g>0 \textrm{ such that } g(z)\le C_g (1+|z|^2) \textrm{ for all } z\in ℝ^N. \label{ConstantForg} 
\end{subnumcases}
In these positions, we first prove the following 
\begin{theorem}[Local existences and extensibility criterion]\label{localSol}
 Let $N\in ℕ$, $\Omega\subset ℝ^N$ be a bounded domain with smooth boundary, $χ\in ℝ$, $δ>0$, $τ\in \{0,1\}$, let $f$ and $g$ comply with \eqref{cond:fandg}.
 For every $u_0\in C^{1+δ}(\Ombar)$, $τv_0\in C^{2+δ}(\Ombar)$ satisfying 
\begin{equation*}
  ∂_{ν}u_0\bdry=0,\tau∂_{ν}v_0\bdry=0 \mbox{ and } u_0\ge 0, τv_0\ge 0,
\end{equation*}
problem \eqref{system} has a unique and nonnegative classical solution 
\begin{equation*}
 (u,v) \in \kl{C^{1,\f{1}2}(\Ombar\times[0,\Tmax))\cap C^{2,1}(\Ombar\times(0,\Tmax))}\times C^{2,1}(\Ombar\times[0,\Tmax))
\end{equation*}
for some maximal $\Tmax\in(0,\infty]$ which is such that for every $p>\f N2$
\begin{equation}\label{extcrit}
   \text{either } \Tmax=\infty \quad \text{or}\quad \limsup_{t\upto \Tmax} \norm[\Lom p]{u(\cdot,t)} = \infty.
\end{equation}
\end{theorem}
With the aid of the above result, we can obtain some more information on solutions to model \eqref{system} if more restrictive assumptions on some data are required. Precisely, we have this further
\begin{theorem}\label{MainTheorem}
Let $N\geq 2$, $\Omega\subset ℝ^N$ be a bounded domain with smooth boundary, $f$ and $g$ comply with \eqref{cond:fandg}, 
$χ> 0$ and $\tau\in \{0,1\}$. Then, for every $u_0\in C^{1+δ}(\Ombar)$, $τv_0\in C^{2+δ}(\Ombar)$ satisfying 
\begin{equation*}
∂_{ν}u_0\bdry=0, \tau∂_{ν}v_0\bdry=0\mbox{ and } u_0\ge 0,\ τv_0\ge 0,
\end{equation*}
whenever 
\begin{equation}\label{ConditionsOnGammaTheo}
 \frac{2N}{N+1} <\gamma \leq 2,
\end{equation}
and 
there are $C_1,C_2,C_3>0$ such that $f$ and $g$ also comply with
\begin{equation}\label{assumptionsCorollary}  
f(s)\leq C_1-C_2s\; \textrm{ for all } s\geq 0 \quad   \textrm{and} \quad   g(z)\geq C_3 |z|^\gamma  \quad \textrm{ for all } z\in \R^N,
\end{equation}
problem \eqref{system} has a unique and nonnegative classical solution 
\begin{align*}
 (u,v) \in& \kl{C^{1,\f{1}2}(\Ombar\times[0,\infty))\cap C^{2,1}(\Ombar\times(0,\infty))\cap L^\infty(\Omega \times (0,\infty))}\\ &\times \kl{C^{2,1}(\Ombar\times[0,\infty))\cap L^\infty(\Omega \times (0,\infty))}, 
\end{align*}
which, in particular, is bounded in the sense that there exists $C>0$ such that 
\[
\|u(\cdot,t)\|_{L^\infty(\Omega)}+\|v(\cdot,t)\|_{W^{1,\infty}(\Omega)}\leq C \quad \textrm{for all } t\geq 0.
\]
\end{theorem}
As a by-product, we can specialize Theorem \ref{MainTheorem} to situations where $f$ and $g$ are external sources exactly reading as in model \eqref{problem}, and thus answer \ref{Question}.
\begin{corollary}\label{Corollary}
Let $N\geq 2$, $\Omega\subset ℝ^N$ be a bounded domain with smooth boundary, $a,b,c,\chi,\delta>0$, $\tau\in \{0,1\}$ and $1\leq \alpha<\beta$. Additionally, let 
relation \eqref{ConditionsOnGammaTheo} hold. 
Then, for every  $u_0\in C^{1+δ}(\Ombar)$, $τv_0\in C^{2+δ}(\Ombar)$ satisfying 
\begin{equation*}
  ∂_{ν}u_0\bdry=0, \tau∂_{ν}v_0\bdry=0\mbox{ and } u_0\ge 0, τv_0\ge 0,
\end{equation*}
problem \eqref{problem} has a unique and nonnegative classical solution, which is bounded.
\end{corollary}
\begin{remark}[Open question]
While the extensibility criterion \eqref{extcrit} gives some indication \textit{how} blow-up would occur (and, in particular, excludes pure gradient blow-up), 
and while there are conditions on $f$ that prevent blow-up even for $g\equiv 0$, we here have to leave open the question whether $g$ alone can prevent unboundedness for small exponents, i.e. whether for some $\gamma  \in \left[1,\frac{2N}{N+1}\right]$ there may exist unbounded solutions to problem \eqref{problem} emanating from certain data. 
\end{remark}

\section{Existence of local-in-time solutions and extensibility criterion. Proof of Theorem \ref{localSol}}\label{LocalSolSection}
While $C_f,C_g,C_1,C_2, C_3$ will always refer to the constants in \eqref{cond:fandg} and \eqref{assumptionsCorollary}, other constants, and in particular the small-letter constants $c_i$, are local to each proof.

The proof of Theorem \ref{localSol} is consequence of this series of lemmas. Until we confirm -- in Lemma~\ref{lem:ex-part2} -- the $C^{2,1}$-regularity of solutions, showing that they are, indeed, classical, ``solution of \eqref{system}'' occasionally must be understood as solution of the weak form of \eqref{system}. The current regularity of $(u,v)$ will always be indicated.
\begin{lemma}\label{lem:ex-part1}
Let $N\in ℕ$, $\Omega\subset ℝ^N$ be a bounded domain with smooth boundary, $δ>0$, $χ\in ℝ$, $τ\in \{0,1\}$, let $f\colon ℝ→ℝ$ and $g\colon ℝ^N\to [0,∞)$ be continuous. 
For every $M_1>0$ and $M_2>0$ there is $T>0$ such that system \eqref{system} for every $u_0\in C^{1+δ}(\Ombar)$, $τv_0\in C^{2+δ}(\Ombar)$ satisfying 
\begin{equation}\label{ex-cond:u0v0}
 \norm[C^{1+δ}(\Ombar)]{u_0}\le M_1, \norm[C^{1+δ}(\Ombar)]{τv_0}\le M_2
 \text{ and  }  ∂_{ν}u_0\bdry=0, 
 \tau∂_{ν}v_0\bdry=0,
\end{equation}
has a solution 
\[
 (u,v)\in C^{1+δ,\f{1+δ}2}(\Ombar\times[0,T])\times C^{2+δ,1+\f{δ}2}(\Ombar\times[0,T]).
\]
\end{lemma} 
\begin{proof}
We fix $T_0>0$. For $R>M_1+1$ and $T\in(0,T_0]$ (to be specified later) we let 
\[
 X_{R,T} = \set{u\in C^1(\Ombar\times[0,T])\mid \sup_{t\in[0,T]}\norm[C^1(\Ombar)]{u(\cdot,t)}\le R} 
\]
which is a closed bounded convex subset of $C^1(\Ombar\times[0,T])$. For $\ubar\in X_{R,T}$ we let $v$ be the solution of 
\begin{equation}\label{ex:v-eq}
 τv_t = Δv - v + \ubar,\;\; ∂_{ν} v = 0,\;\; τv(\cdot,0)=τv_0
\end{equation}
(whose existence as classical solution and uniqueness follow from, e.g., \cite[IV.5.3]{LSU} if $τ=1$ and from \cite[6.31]{GT} if $τ=0$), and define $Φ(\ubar)=u\in L^2((0,T);W^{1,2}(\Om))\cap C([0,T];\Lom2)$ to be the weak solution of 
\begin{equation}\label{ex:u-eq}
 u_t = Δu - \chi ∇\cdot(u∇v) + f(\ubar) - g(∇\ubar),\;\;  ∂_{ν} u = 0\;\;  u(\cdot,0)=u_0
\end{equation}
(whose existence follows by \cite[III, §5]{LSU}, uniqueness due to \cite[III.3.1]{LSU}). 

There is $c_1=c_1(M_2,R)>0$ such that every solution of \eqref{ex:v-eq} (for any $\ubar \in X_{R,T}$ and $v_0$ with \eqref{ex-cond:u0v0}) satisfies 
\[
 \sup_{t\in[0,T]}\norm[C^{1+δ}(\Ombar)]{v(\cdot,t)}\le c_1
\]
(cf. \cite[Thm.~1.2]{lieberman_paper} if $τ=1$ or \cite[Thm 6.30]{GT} if $τ=0$) and there is $c_2=c_2(M_1,R,c_1, f, T_0)>0$ such that (for any $u_0$ as in \eqref{ex-cond:u0v0} and any $\ubar \in X_{R,T}$) the solution $u$ of \eqref{ex:u-eq} satisfies
\[
 \norm[L^\infty(\Om\times(0,T_0))]{u}\le c_2, 
\]
because from 
\begin{align*}
u_t &\le Δu - χ\nabla\cdot(uh) + \sup_{s\in[-R,R]} f(s)&&\text{ and } &
u_t &\ge Δu - χ\nabla\cdot(uh) + \inf_{s\in[-R,R]}f(s) - \sup_{z\in B_R(0)} g(z) 
\end{align*}
with the bounded function $h=∇v$, we obtain upper and lower bounds for $u$ from semigroup estimates like \cite[Lemma~1.3(iv)]{win_aggregation_vs}.
Having ensured 
boundedness of $u$, we can now employ \cite[Thm.~1.2]{lieberman_paper} to see that there is 
$c_3>0$ such that (for any $u_0$ as in \eqref{ex-cond:u0v0} and any $\ubar \in X_{R,T}$) the solution $u$ of \eqref{ex:u-eq} satisfies
\[
 \normm{C^{1+δ,\f{1+δ}2}(\Ombar\times[0,T])}{u}\le c_3. 
\]
Consequently, 
\[
 \norm[C^1(\Ombar)]{u(\cdot,t)} \le \norm[C^1(\Ombar)]{u_0} + \norm[C^1(\Ombar)]{u(\cdot,t)-u_0} \le M_1 + c_3 T^{δ}, 
\]
which shows that for $T<c_3^{-\f1{δ}}$, $Φ$ maps $X_{R,T}$ to $X_{R,T}$. Moreover, $Φ(X_{R,T})$ is bounded in $C^{1+δ,\f{1+δ}2}(\Ombar\times[0,T])$, hence compact in $C^1(\Ombar\times[0,T])$, and $\Phi$ is continuous. 
Schauder's fixed point theorem asserts the existence of a fixed point $u\in Φ(X_{R,T})$ of $Φ$.  
\end{proof}

\begin{lemma}\label{LemmaSignUandV}
 In addition to the assumptions of Lemma~\ref{lem:ex-part1}, assume $u_0\ge 0$ and let $f$ and $g$ fulfill \eqref{fAssTheo} and \eqref{GlocallyLips}. Then $u\ge 0$.
\end{lemma}
\begin{proof}
 Since $0$ is a subsolution due to $f(0)\ge 0$ and $g(0)=0$, this follows from a classical comparison principle (\cite[52.6]{QS}), applied to the first equation of \eqref{system} for fixed (and $C^2(\Ombar)$-bounded) $v$. 
\end{proof}

\begin{lemma}\label{lem:unique}
Let $\Omega$, $τ$, $χ$, $u_0$, $τv_0$, $f$, $g$ be given as in Lemma~\ref{lem:ex-part1}, and additionally let $f$ and $g$ be locally Lipschitz continuous.
Let $(u_1,v_1), (u_2,v_2) \in C^0(\Ombar\times[0,T])$ satisfy $∇u_1,∇u_2,∇v_1,∇v_2\in C^0(\Ombar\times[0,T])$, $(u_1)_t, (u_2)_t\in C^0(\Ombar\times[0,T])$ (and, if $τ=1$, also $(v_1)_t, (v_2)_t \in C^0(\Ombar\times[0,T])$) and be solutions of \eqref{system}. 
Then $(u_1,v_1)=(u_2,v_2)$ in $\Ombar\times[0,T]$.
\end{lemma}
\begin{proof}
 By means of a Grönwall-type argument this is obtained from a study of \[\f{d}{dt}\kl{ \io(u_1-u_2)^2+τ\io(v_1-v_2)^2}\quad \textrm{for all } t \in [0,T].\qedhere\]
\end{proof}

\begin{lemma}\label{lem:ex-part2}
Let $N\in ℕ$, $\Omega\subset ℝ^N$ be a bounded domain with smooth boundary, $δ>0$, $\chi \in \R$, $τ\in \{0,1\}$, and let $f$ and $g$ fulfill assumptions \eqref{fAssTheo} and \eqref{GlocallyLips}. For $u_0\in C^{1+δ}(\Ombar)$, $τv_0\in C^{2+δ}(\Ombar)$ satisfying 
\begin{equation}\label{cond:init-compatibility-nonnegativity}
  ∂_{ν}u_0\bdry=0, \tau∂_{ν}v_0\bdry=0\mbox{ and } u_0\ge 0, τv_0\ge 0
\end{equation}
there is $\TM\in(0,\infty]$ such that \eqref{system} has a unique solution
 \[
 (u,v) \in \kl{C^{1,\f{1}2}(\Ombar\times[0,\TM))\cap C^{2,1}(\Ombar\times(0,\TM))}\times C^{2,1}(\Ombar\times[0,\TM)),
 \]
 which is nonnegative, 
 and such that 
 \[
  \text{either } \TM=\infty \quad \text{or}\quad \lim_{t\upto \TM} \kl{\norm[C^{1+δ}(\Ombar)]{u(\cdot,t)}+\norm[C^{1+δ}(\Ombar)]{v(\cdot,t)}} = \infty. 
 \] 
\end{lemma}
\begin{proof}
Repeated application of the local existence result of Lemma~\ref{lem:ex-part1}, aided by the uniqueness statement of Lemma~\ref{lem:unique}, ensures the existence of a solution  $(u,v)\in C^{1+δ,\f{1+δ}2}(\Ombar\times[0,\TM))\times C^{2+δ,1+\f{δ}2}(\Ombar\times[0,\TM))$ with some maximally chosen $\TM>0$. If $\TM$ were finite and we did not have $\lim_{t\upto \TM} \kl{\norm[C^{1+δ}(\Ombar)]{u(\cdot,t)}+\norm[C^{1+δ}(\Ombar)]{v(\cdot,t)}} = \infty $, we could find $M>0$ and a sequence $t_k\upto \TM$ such that 
 $\kl{\norm[C^{1+δ}(\Ombar)]{u(\cdot,t_k)}+\norm[C^{1+δ}(\Ombar)]{v(\cdot,t_k)}}\le M$ for all $k\in ℕ$, which would enable us to extend the solution on $(t_k,t_k+T(M))$ with $T(M)\coloneq T(M,M)$ taken from Lemma~\ref{lem:ex-part1}. As $t_k+T(M)>\TM$ for large $k$, this would contradict maximality of $\TM$. 
 
Further application of classical regularity theory (cf. the Schauder estimates contained in \cite[IV.5.3]{LSU}) ensures that also $u\in C^{2,1}(\Ombar\times(0,\TM))$. 

Finally, we observe that from Lemma \ref{LemmaSignUandV} we have that $u \geq 0$ and, in turn,  $v\geq 0$ is consequence of the elliptic or parabolic maximum principle applied to the second equation of \eqref{system}.  
\end{proof}

The presence of additional gradient terms becomes noticable in the next step. In order to rely on classical regularity theory, we require some restriction of their growth, so as to preclude gradient blow-up.
\begin{lemma}\label{lem:ex:continuation_in_conedelta}
 Let $\Om$, $\chi$, $τ$, $u_0, v_0$, $f$ and $g$ be as in Lemma~\ref{lem:ex-part2}, and let $g$ satisfy assumption \eqref{ConstantForg}. Then there is $\TM\in(0,\infty]$ such that \eqref{system} has a solution 
 \[
  (u,v) \in \kl{C^{1,\f{1}2}(\Ombar\times[0,\TM))\cap C^{2,1}(\Ombar\times(0,\TM))}\times C^{2,1}(\Ombar\times[0,\TM))
 \]
 and such that 
 \[
  \text{either } \TM=\infty \quad \text{or}\quad \limsup_{t\upto \TM} \norm[\Lom\infty]{u(\cdot,t)} = \infty 
 \]
\end{lemma}
\begin{proof}
 Assuming that the maximal existence time $\TM$ (cf. Lemma~\ref{lem:ex-part2}) were finite, by the extensibility criterion in Lemma~\ref{lem:ex-part2} it is sufficient to show that boundedness of $\norm[\Lom\infty]{u(\cdot,t)}$ for $t\in[0,\TM]$ entails boundedness of $\norm[C^{1+δ}(\Ombar)]{u(\cdot,t)}+\norm[C^{1+δ}(\Ombar)]{v(\cdot,t)}$ for $t\in(\f{\TM}2,\TM)$. Boundedness of $\norm[C^{1+δ}(\Ombar)]{v(\cdot,t)}$ results from that of $\norm[\Lom\infty]{u(\cdot,t)}$ by means of \cite[Thm.~1.2]{lieberman_paper} (if $τ=1$) or, e.g. from \cite[Thm.~I.19.1]{friedman_pde} (if $\tau=0$). According to the nomenclature in  \cite{lieberman_paper} letting $A(x,t,z,p)=p-z\nabla v(x,t)$ and $B(x,t,z,p)=f(z)-g(p)$, we conclude from the fact that $u$ is a bounded solution to $u_t=\nabla\cdot(A(x,t,u,∇u)) + B(x,t,u,∇u)$ and \cite[Thm.~1.2]{lieberman_paper} that $\sup_{t\in(\f{3\TM}4,\TM)}\norm[C^{1+δ}(\Ombar)]{u(\cdot,t)}$ is finite.
\end{proof}

\begin{lemma}\label{lem:extcrit}
Let $N\in ℕ$, $\Omega\subset ℝ^N$ be a bounded domain with smooth boundary, $δ>0$, $\chi \in \R$, $τ\in \{0,1\}$, let $f$ and $g$ comply with \eqref{fAssTheo}, \eqref{ConstantForfThe}  and \eqref{GlocallyLips}.
  For every $p>\f N2$ and $K>0$, there are $q>N$, $L>0$ and $M>0$ such that if 
  $(u,v)$ is a solution to \eqref{system} for any $u_0\in C^{1+δ}(\Ombar)$, $τv_0\in C^{2+δ}(\Ombar)$ satisfying \eqref{cond:init-compatibility-nonnegativity} and
 \[
  \norm[\Lom p]{u(\cdot,t)} \le K \quad \text{for all } t\in [0,\TM),\quad \norm[\Lom \infty]{u_0}\le K \text{ and } \norm[\Lom q]{∇v_0}\le K, 
 \]
 then 
 \begin{equation}\label{eq:est_na_v}
  \norm[\Lom q]{∇v(\cdot,t)} \le L \text{ for all } t\in(0,\TM)
 \end{equation}
 and 
 \[
  \norm[\Lom \infty]{u(\cdot,t)}\le M \text{ for all } t\in(0,\TM).
 \]
\end{lemma}
\begin{proof}
 From $\norm[\Lom p]{u(\cdot,t)} \le K$ and $p>\f N2$, which implies that $\f{Np}{(N-p)_{+}}>N$, by elliptic regularity theory \cite[I.19.1]{friedman_pde} (if $τ=0$) or (if $τ=1$) semigroup estimates (\cite[Lemma~1.3(ii)]{win_aggregation_vs} applied in $∇v(t)=∇e^{t(Δ-1)}v_0 +\int_0^t ∇e^{(t-s)(Δ-1)}u(s)ds$) we find $L>0$ and $q>N$ such that \eqref{eq:est_na_v} holds. 
 We let $c_1>0$ and $c_2>0$ be the constants from the semigroup estimates in \cite[Lemma~1.3(i)]{win_aggregation_vs} and \cite[Lemma~1.3(iv)]{win_aggregation_vs}, 
  and use that $f(s)\le C_f$ for all $s\in ℝ$. Choosing $\tilde p>p$ so large that $(\f1q+\f1{\tilde p})^{-1}>N$, we find that 
 \begin{align*}
  0&\le u(t)\le  e^{tΔ} u_0 -\chi \int_0^t e^{(t-s)Δ}∇\cdot(u∇v) ds + \int_0^t e^{(t-s)Δ}f(u(s))ds\\
  &\le \norm[\Lom\infty]{u_0} + c_2 |χ|\int_0^t (1+(t-s)^{-α}) \norm[\Lom p]{u(\cdot,s)}^{\f{p}{\tilde p}}\norm[\Lom\infty]{u(\cdot,s)}^{1-\f{p}{\tilde p}}\norm[\Lom q]{∇v(\cdot,s)} ds + C_f
 \end{align*}
 for $t\le \min\set{1,\TM}$ and $α=\f12+\f N2(\f1q+\f1{\tilde p})<1$ and that with $τ=t-1+s$
 \begin{align*}
  0\le &u(\cdot,t) \le e^{1Δ}u(\cdot,t-1)-\chi \int_0^1 e^{(1-s)Δ}∇\cdot (u(\cdot,τ)∇v(\cdot,τ)) ds +\int_0^1e^{(1-s)Δ}f(u(\cdot,τ)) ds\\
  \le &2c_1\norm[\Lom p]{u(\cdot,t-1)} \\
  &+ c_2 |\chi|\int_0^1 (1+(1-s)^{-α}) \norm[\Lom p]{u(\cdot,τ)}^{\f{p}{\tilde p}}\norm[\Lom\infty]{u(\cdot,τ)}^{1-\f{p}{\tilde p}}\norm[\Lom q]{∇v(\cdot,τ)} ds + C_f
 \end{align*}
 for $t\in(1,\TM)$, so that upon defining $M(t)=\sup_{τ\in(0,t)} \norm[\Lom\infty]{u(\cdot,τ)}$, we obtain 
 \[
  M(t) \le \underbrace{C_f+\max\set{\norm[\Lom\infty]{u_0},2c_1 K}}_{=:c_3} 
  + \underbrace{c_2|χ|\int_0^1(1+τ^{-\f12-\f N2(\f1q+\f1{\tilde p})}) dτ K^{\f p{\tilde p}} L}_{=:c_4}(M(t))^{1-\f p{\tilde p}},
 \]
 for any $t\in(0,\TM)$, which shows that 
 \[
  M(t)\le \sup\set{x\mid x\le c_3+c_4 x^{1-\f{p}{\tilde p}}} =: M <\infty.\qedhere
 \]
\end{proof}

\begin{proof}[Proof of Theorem \ref{localSol}]
 The existence and the nonnegativity follow from Lemma~\ref{lem:ex-part2}; this form of the extensibility criterion is obtained from Lemma~\ref{lem:extcrit} in combination with Lemma \ref{lem:ex:continuation_in_conedelta}.
\end{proof}

\section{Some a priori estimates. Proofs of Theorem \ref{MainTheorem} and Corollary \ref{Corollary}}
The proof of Theorem \ref{MainTheorem} goes through some steps, which altogether ensure uniform boundedness of $u$ in some $L^p$-norm.  
\subsection{\textit{A priori} estimates}
From now to $\S$\ref{SusectionProofTheo}, we assume that $u_0,τv_0$ are given as required by Theorem~\ref{MainTheorem} and, for $\chi>0$, by $(u,v)$ we will indicate the local solution to model \eqref{system} provided by Theorem \ref{localSol}; moreover, with the aim of exploiting the related extensibility criterion \eqref{extcrit}, let us focus on the derivation of some \textit{a priori} $L^p$-estimates for the $u$-component, by specifying how each of the constraints in Theorem \ref{MainTheorem} takes part in our computations. 
In a first step, let us derive a control on the total mass of $u$ from nonnegativity of $g$ and the behaviour of $f$. 
\begin{lemma}\label{LemmaMassAndNablavRegularity}
Let $f$ and $g$ satisfy \eqref{assumptionsCorollary}. Then, for some $m_0>0$ we have that
$u$ is such that 
\begin{equation}\label{massConservation}
\int_\Omega u(x, t)dx \leq m_0\quad \textrm{for all }\, t \in \left(0,T_{max}\right).
\end{equation}
\end{lemma}
\begin{proof}
An integration over $\Omega$ of the first equation in \eqref{system} due to the estimates in \eqref{assumptionsCorollary} shows that 
\[
\frac{d}{dt} \int_\Omega u =\int_\Omega (f(u)-g(\nabla u))\leq  C_1|\Omega|-C_2\int_\Omega u \quad \textrm{in }\, \left(0,T_{max}\right),
\]
so that we arrive at $\int_\Omega u(\cdot,t)\leq m_0=\max\left\{\int_\Omega u_0,\frac{C_1|\Omega|}{C_2}\right\}$ for all $t\in (0,\TM)$.  
\end{proof}

The term $\int_\Omega |\nabla u^\frac{p}{2}|^2$ in any dimension controls $\int_\Omega u^p$:
\begin{lemma}\label{lem:gni}
Let $p>1$. Then there is $k>0$ such that 
\begin{equation*}\label{utopControllGrad}
\int_\Omega u^p \leq 4\frac{p-1}{p} \int_\Omega |\nabla u^\frac{p}{2}|^2 +k \quad \textrm{on }\; (0,T_{max}).
\end{equation*}
\end{lemma}
\begin{proof}
We let 
$$
\check{\theta}=\frac{\frac{p}{2}-\frac{1}{2}}{\frac{p}{2}-\frac{1}{2}+\frac{1}{N}}\in \left(0,1\right),
$$ 
and apply the Gagliardo-Nirenberg inequality (possibly used in the less common version with exponents below one, given in, e.g. \cite[Lemma 2.3]{LiLankeitNonlinearity2016Hapto}) and 
\begin{equation}\label{InequalitiAlgebraic2ToThePower}
(A+B)^l\le \max\{1,2^{l-1}\}(A^l+B^l),\quad l>0, A,B\geq 0,
\end{equation}
which together with \eqref{massConservation} yields $\check{C}_{GN}>0$ and $k>0$ such that 
\begin{equation*}
\int_\Omega u^p=\lVert u^\frac{p}{2}\rVert_{L^2(\Omega)}^2 \leq \check{C}_{GN}\left(\lVert \nabla u^\frac{p}{2}\rVert_{L^2(\Omega)}^{2\check{\theta}} \lVert u^\frac{p}{2}\rVert_{L^\frac{2}{p}(\Omega)}^{2(1-\check{\theta})}+\lVert u^\frac{p}{2}\rVert_{L^\frac{2}{p}(\Omega)}^2\right)\leq 4\frac{p-1}{p} \int_\Omega |\nabla u^\frac{p}{2}|^2 +k
\end{equation*}
on $(0,\Tmax)$.
\end{proof}

\begin{lemma}\label{lem:ddtioup}
Let $f$ and $g$ satisfy \eqref{assumptionsCorollary} and let $p> 1$. There is $K>0$ such that 
 \begin{align*}
  \frac{d}{dt}&\int_\Omega u^p + \int_\Omega u^p\\
  &\le  - \chi (p-1)\int_\Omega u^p\Delta v 
  + \chi (p-1)\io u^{p+1} 
  + K -C_3 p \left(\frac{p-1+\gamma}{\gamma}\right)^{-\gamma}\int_\Omega |\nabla u^\frac{p-1+\gamma}{\gamma}|^\gamma
 \end{align*}
 on $(0,\Tmax)$.
\end{lemma}
\begin{proof}
 Standard testing procedures, together with the estimates on $f$ and $g$ from \eqref{assumptionsCorollary}, give
\begin{equation*}
\begin{split}
\frac{d}{dt}\int_\Omega u^p&=p\int_\Omega u^{p-1}\left(\Delta u-\chi\nabla \cdot \left(u \nabla v \right)+f(u)-g(\nabla u)\right)\\
&\leq -4\frac{p-1}{p}\int_\Omega |\nabla u^\frac{p}{2}|^2 -\chi (p-1)\int_\Omega u^{p}\Delta v+C_1p\int_\Omega u^{p-1}-C_3 p \int_\Omega u^{p-1}|\nabla u|^\gamma
\end{split}
\end{equation*}
on $(0,\Tmax)$, and straightforward manipulation of the term $-C_3p u^{p-1}|\nabla u|^\gamma$ results in the claim, if we take into account that by Young's inequality with some $c_1>0$ we have
\begin{equation*}\label{YoungForFofU}
C_1p\int_\Omega u^{p-1}\leq \chi (p-1)\int_\Omega u^{p+1} +c_1 \quad \textrm{on }\; (0,T_{max}),
\end{equation*}
and employ Lemma~\ref{lem:gni}.
\end{proof} 
Suitably estimating $\io u^{p+1}$ will use the largeness assumption on $\gamma$ given in \eqref{ConditionsOnGammaTheo}:
\begin{lemma}\label{lem:up+1controledGrad}
 Let $\frac{2N}{N+1}<\gamma\leq 2$. Then there is $p>\f{N}2$ such that for any $K>0$ there is $c=c(K,p)>0$ such that 
 \[
  K\int_\Omega u^{p+1} \leq C_3p \left(\frac{p-1+\gamma}{\gamma}\right)^{-\gamma}\int_\Omega |\nabla u^\frac{p-1+\gamma}{\gamma}|^\gamma +c \quad \textrm{on }\; (0,T_{max}).
 \]
\end{lemma}
\begin{proof}
Due to the assumption on $γ$, it is possible to choose $p>\f{N}2$ such that 
\begin{equation}\label{deftheta}
\theta=\frac{\frac{p-1+\gamma}{\gamma}(1-\frac{1}{p+1})}{\frac{1}{N}-\frac{1}{\gamma}+\frac{p-1+\gamma}{\gamma}}\in \left(0,1\right)\quad \textrm{and}\quad
\frac{\theta(p+1)}{p-1+\gamma}\in \left(0,1\right).
\end{equation}
By \eqref{deftheta}, the Gagliardo--Nirenberg and Young inequalities as in Lemma \ref{lem:gni}, in conjunction with
\eqref{massConservation} for any $K>0$, 
provide also some $C_{GN},c>0$ such that
\begin{equation*}
\begin{split}
K\int_\Omega u^{p+1}&=K\lVert u^\frac{p-1+\gamma}{\gamma}\rVert_{L^\frac{\gamma(p+1)}{p-1+\gamma}(\Omega)}^\frac{\gamma(p+1)}{p-1+\gamma}\\
&\leq K C_{GN}\left(\lVert \nabla u^\frac{p-1+\gamma}{\gamma}\rVert_{L^\gamma(\Omega)}^\frac{\theta\gamma(p+1)}{p-1+\gamma} \lVert u^\frac{p-1+\gamma}{\gamma}\rVert_{L^\frac{\gamma}{p-1+\gamma}(\Omega)}^\frac{(1-\theta)\gamma(p+1)}{p-1+\gamma}+\lVert u^\frac{p-1+\gamma}{\gamma}\rVert_{L^\frac{\gamma}{p-1+\gamma}(\Omega)}^\frac{\gamma(p+1)}{p-1+\gamma}\right)\\ 
& \leq C_3p \left(\frac{p-1+\gamma}{\gamma}\right)^{-\gamma}\int_\Omega |\nabla u^\frac{p-1+\gamma}{\gamma}|^\gamma +c \quad \textrm{on }\; (0,T_{max}).\qedhere
\end{split}
\end{equation*}
\end{proof}

For the remaining study toward the uniform boundedness of $\int_\Omega u^p$, we will distinguish the cases $\tau=0$ and $\tau=1$.
\subsection{Analysis of the parabolic-elliptic case\texorpdfstring{: $\tau=0$}{}}
\begin{lemma}\label{LemmaEst1}
Let $f$ and $g$ satisfy \eqref{assumptionsCorollary} and let $\frac{2N}{N+1}<\gamma\leq 2$. Then  there exists $L_0>0$ such that for some $p>\frac{N}{2}$
\begin{equation*}
\int_\Omega u^p\leq L_0 \quad \textrm{for all} \quad t \in \left(0,T_{max}\right).
\end{equation*}
\end{lemma}
\begin{proof}
We take $p$ from Lemma~\ref{lem:up+1controledGrad}. By explicitly using that $-u^pΔv=u^p(u-v)\le u^{p+1}$ on $\Omega \times (0,\TM)$, from Lemma~\ref{lem:ddtioup} we find that 
\begin{equation}\label{TestProcedureEll}
\frac{d}{dt}\int_\Omega u^p + \int_\Omega u^p \leq 
2\chi (p-1)\int_\Omega u^{p+1}-C_3 p \left(\frac{p-1+\gamma}{\gamma}\right)^{-\gamma}\int_\Omega |\nabla u^\frac{p-1+\gamma}{\gamma}|^\gamma+c_1
\end{equation}
on $(0,\Tmax)$.
Finally by taking into account relation \eqref{TestProcedureEll} and Lemma \ref{lem:up+1controledGrad} with $K=2\chi(p-1)$, we are given some $c_2>0$ fulfilling
\begin{equation*}
\frac{d}{dt}\int_\Omega u^p+\int_\Omega u^p\leq c_2\quad \textrm{on }\; \left(0,T_{max}\right),
\end{equation*}
which combined with $(\int_\Omega u^p)_{\mid_{t=0}}  =\int_\Omega u_0^p$ concludes the proof through comparison reasoning for ordinary differential inequalities. 
\end{proof}

\subsection{Analysis of the parabolic-parabolic case\texorpdfstring{: $\tau=1$}{}}
In the fully parabolic case, it is no longer possible to replace $Δv$ easily by inserting $u-v$; in the spirit of \cite{IshidaYokotaDCDS-SViaMSR,IshidaYokotaJDE-2012SmallData,SenbaSuzukiAAN2006} we will therefore rely on Maximal Sobolev regularity for the second equation, which will enable us to properly deal with integral terms arising in testing procedures.

As the essential tool, we thus first recall the following consequence of Maximal Sobolev regularity results (like \cite{hieber_pruess} or \cite[Thm. 2.3]{giga_sohr}):
 \begin{lemma}\label{lem:MaxReg}
  Let $N\in ℕ$, $\Omega\subset ℝ^N$ be a bounded domain with smooth boundary, $q\in(1,\infty)$, $μ>\f1{q}$ and $K>1$. Then there is $C_{MR}>0$ such that the following holds: Whenever $T\in(0,\infty]$, $I=[0,T)$, $f\in L^q(I;L^q(\Omega))$ and $v_0\in W^{2,q}(\Omega)$, with $∂_νv_0=0$ on $\partial\Omega$, is such that $\norm[W^{2,q}(\Om)]{v_0}\le K$, every solution $v\in W_{loc}^{1,q}(I;L^q(\Om))\cap L^q_{loc}(I;W^{2,q}(\Om))$ of 
 \[
  v_t=Δv-μv + f\;\; \text{ in }\;\;\Om\times(0,T);\quad
  \partial_{ν} v=0\;\; \text{ on }\;\;\partial\Omega \times(0,T); \quad v(\cdot,0)=v_0 \;\; \text{ on }\;\;\Omega  
 \]
 satisfies 
 \[
  \int_0^t e^s \kl{\io |Δv(\cdot,s)|^q}ds \le C_{MR} \left[1+\int_0^t e^s\kl{\io |f(\cdot,s)|^{q}}ds\right] \quad \text{for all } t\in(0,T).
 \]
 \end{lemma}
\begin{proof}
For $(x,t)\in \Omega\times (0,T)$, let us set $z(x,t):=e^{\frac{t}{q}}v(x,t)$. 
Then easy computations establish that $z$ solves 
\begin{equation*}
\begin{cases}
 z_t=\Delta z -\left(μ-\frac{1}{q}\right)z +e^\frac{t}{q}f & \textrm{in } \Omega \times (0,T),\\
\partial_\nu z=0 & \textrm{on } \partial \Omega \times (0,T),\\
 z(x,0)= v_{0}(x) & x\in  \Omega.
\end{cases}
\end{equation*}
Subsequently, let us apply Maximal Sobolev regularity (\cite[(3.8)]{hieber_pruess}, \cite[Thm. 2.3]{giga_sohr}) to $A=Δ-(μ-\f1q)$, $X=L^q(\Omega)$ and $X_1=D(A)=W^{2,q}_\mathcal{\frac{\partial }{\partial \nu}}(\Omega)= \{w\in W^{2,q}(\Omega): \frac{\partial w}{\partial \nu}=0\,\,\textrm{on}\;\partial \Omega\}$, which asserts that with some $c_1>0$ we have for every $t\in(0,T)$ that
\begin{equation*}
\begin{split}
 \lVert Δz\rVert _{L^q([0,t];L^q(\Omega))} +\lVert z_{t}\rVert_{L^q([0,t];L^{q}(\Omega))}
\leq c_1 \kl{\lVert v_{0}\rVert_{1-\frac{1}{q},q}+
\Big(\int_0^t\|e^\frac{s}{q}f(\cdot,s)\|_{L^q(\Omega)}^q\, ds\Big)^\frac{1}{q}},
\end{split}
\end{equation*}
where $\lVert \cdot\rVert_{1-\frac{1}{q},q}$ represents the norm in the interpolation space $(X,X_1)_{1-\frac{1}{q},q}$, which can be estimated by $K$, according to the assumption on $v_0$. 
In turn, we have that for $C_{MR}=c_1^q2^{q-1}K^q$ (cf. \eqref{InequalitiAlgebraic2ToThePower}) and  all $t\in(0,T)$
\begin{align}\label{inequalityPrussBIS}
\int_0^t \Big(\int_\Omega |\Delta z(\cdot,s)|^{q}\Big)\,ds
&\leq c_1^{q}
\left[K+\Big(\int_0^t\|e^\frac{s}{q}f(\cdot,s)\|_{L^{q}(\Omega)}^{q}\, ds\Big)^\frac{1}{q}\right]^{q}\\
&\leq C_{MR} \left[1+\left(\int_0^t e^s \Big(\int_\Omega |f(\cdot,s)|^{q} \right)ds \right].\nn
\end{align}
We can finally obtain the claim by re-substituting $z(\cdot,t):=e^{\frac{t}{q}}v(\cdot,t)$ into relation \eqref{inequalityPrussBIS}. 
\end{proof}

With the aid of the above lemma, let us continue with the following $L^p$ bound.
\begin{lemma}\label{LemmaEst3}
Let $f$ and $g$ satisfy \eqref{assumptionsCorollary} and let $\frac{2N}{N+1}<\gamma\leq 2$. Then  there exists $L_1>0$ such that for some $p>\frac{N}{2}$
\begin{equation*}
\int_\Omega u^p\leq L_1 \quad \textrm{for all} \quad t \in \left(0,T_{max}\right).
\end{equation*}
\begin{proof}
Again we let $p>\frac{N}2$ be as in Lemma~\ref{lem:up+1controledGrad}. From Lemma \ref{lem:ddtioup} we have  
that 
\[
  \frac{d}{dt}\int_\Omega u^p + \int_\Omega u^p \le  \chi (p-1)\int_\Omega |u^p\Delta v|+ \chi (p-1)\io u^{p+1} 
  + c_1 -C_3 p \left(\frac{p-1+\gamma}{\gamma}\right)^{-\gamma}\int_\Omega |\nabla u^\frac{p-1+\gamma}{\gamma}|^\gamma
\]
on $(0,\Tmax)$. By invoking Young's inequality, we infer that on $(0,T_{max})$,
 \[
   \frac{d}{dt}\int_\Omega u^p + \int_\Omega u^p \le  \chi \scalemath[0.8]{\frac{(p-1)(2p+1)}{p+1}} \int_\Omega u^{p+1} + \chi\scalemath[0.8]{\frac{p-1}{p+1}} \int_\Omega|\Delta v|^{p+1}
   + c_1 -C_3 p \scalemath[0.8]{\left(\frac{p-1+\gamma}{\gamma}\right)}^{-\gamma}\int_\Omega |\nabla u^\frac{p-1+\gamma}{\gamma}|^\gamma. 
 \]
Successively, solving the above ordinary differential inequality for $\int_\Omega u(\cdot,t)^p$, 
we can derive  for all $t\in (0,\TM)$, 
\begin{align*}\label{FirstIneqForY}
&e^t\int_\Omega u^p 
\leq 
    \|u_0\|_{L^p(\Omega)}^p\\
    &+\int_0^t e^s\left[\chi \scalemath[0.8]{\frac{(p-1)(2p+1)}{p+1}}\int_\Omega u^{p+1}
    +\chi\frac{p-1}{p+1}\int_\Omega |\Delta v|^{p+1}
    -C_3 p \scalemath[0.8]{\left(\frac{p-1+\gamma}{\gamma}\right)}^{-\gamma}\int_\Omega |\nabla u^\frac{p-1+\gamma}{\gamma}|^\gamma+c_1\right]ds.
\end{align*}
In estimating the integral containing $Δv$, let us now exploit Lemma~\ref{lem:MaxReg} with $f=u$, $q=p+1$, $μ=1$ to obtain $C_{MR}>0$ such that for all $t\in(0,\TM)$
\begin{equation*}\label{ThirdIneqForY}
e^t\int_\Omega u^p 
\leq 
    c_2
    +\int_0^t e^s\left[c_3\int_\Omega u^{p+1}
    -C_3 p \left(\frac{p-1+\gamma}{\gamma}\right)^{-\gamma}\int_\Omega |\nabla u^\frac{p-1+\gamma}{\gamma}|^\gamma\right]ds+c_4(e^t-1),
\end{equation*}
where $c_2=\|u_0\|_{L^p(\Omega)}^p$, $c_3=\chi \frac{(2p+1)(p-1)}{p+1}+\chi\frac{p-1}{p+1}C_{MR}$ and $c_4=c_1+\frac{p-1}{p+1}\chi C_{MR}$. 
In this way, we again invoke Lemma~\ref{lem:up+1controledGrad}, applied with $K=c_3$, so as to rephrase the previous inequality as
\begin{equation*}
e^t\int_\Omega u^p \leq c_2+c_4(e^t-1)\quad \textrm{for all } t \in (0,\TM),
\end{equation*}
from which we obtain the assertion with $L_1:=c_2+c_4$.
\end{proof}
\end{lemma}
Finally  we are in a position to conclude.

\subsection{Final proofs of boundedness and global existence}\label{SusectionProofTheo}
\begin{proof}[Proof of Theorem~\ref{MainTheorem}]
By Lemma~\ref{lem:extcrit} and Theorem \ref{localSol}, whenever the first component of the local solution $(u,v)$ to problem \eqref{system} belongs to $L^{\infty}((0,\TM); L^p(\Omega))$ for some $p>\frac{N}{2}$, then actually $u\in L^{\infty}(\Omega \times (0,\infty))$;  evidently, by using the equation for $v$, if $u$ lies in $L^{\infty}(\Omega \times (0,\infty))$, so does $v$. Subsequently, the claim follows for $\gamma>\frac{2N}{N+1}$ and $\tau=0$ and $\tau=1$, by invoking Lemmas \ref{LemmaEst1} and \ref{LemmaEst3}, respectively.
\end{proof}
\begin{proof}[Proof of Corollary~\ref{Corollary}]
Let us observe that for any $a,b,C_2>0$ and $β>α\ge 1$, the function defined by $f(s)=as^{α}-bs^{β}$ is continuously differentiable, so that \eqref{fAssTheo} holds, and satisfies $\lim_{s\to\infty} (f(s)+C_2s)=-\infty$, so that there is $C_1>0$ such that $f(s)+C_2s\le C_1$ for every $s\ge 0$. In particular, $f$ satisfies \eqref{ConstantForfThe} and \eqref{assumptionsCorollary}. Similarly, for any $γ\in[1,2]$ and $c>0$, $g\colon ℝ^N\to ℝ$, $g(z)=c|z|^{γ}$ obeys \eqref{GlocallyLips}, \eqref{ConstantForg} and \eqref{assumptionsCorollary}. We therefore can apply Theorem \ref{MainTheorem}.
\end{proof}
\subsection*{Acknowledgments}
GV is member of the Gruppo Nazionale per l'Analisi Matematica, la Probabilit\`a e le loro Applicazioni (GNAMPA) of the Istituto Nazionale di Alta Matematica (INdAM) and is partially supported by the research projects \emph{Evolutive and Stationary Partial Differential Equations with a Focus on Biomathematics} (2019, Grant Number: F72F20000200007),  {\em  Analysis of PDEs in connection with real phenomena} (2022), funded by  \href{https://www.fondazionedisardegna.it/}{Fondazione di Sardegna} (2021, Grant Number: F73C22001130007) and by MIUR (Italian Ministry of Education, University and Research) Prin 2017 \emph{Nonlinear Differential Problems via Variational, Topological and Set-valued Methods} (Grant Number: 2017AYM8XW) 
and by GNAMPA-INdAM Project \emph{Equazioni differenziali alle derivate parziali nella modellizzazione di fenomeni reali}
(CUP\_E53C22001930001). SI is supported by JSPS KAKENHI Grant Number JP21K13815.

\footnotesize 
  \setlength{\parskip}{0pt}
  \setlength{\itemsep}{0pt plus 0.2ex}

\def\cprime{$'$}

\end{document}